\documentclass[twoside, a4paper, 12pt]{amsart}
\usepackage{amsaddr}
\usepackage{fullpage}
\usepackage{amsfonts}
\usepackage{amssymb}
\usepackage{enumitem}
\theoremstyle{plain}
\newtheorem*{claim*}{Claim}
\newtheorem{thm}{Theorem}[section]
\newtheorem{cor}[thm]{Corollary}
\newtheorem{lemma}[thm]{Lemma}
\newtheorem{prop}[thm]{Proposition}
\theoremstyle{definition}

\newtheorem{ex}[thm]{Example}

\newtheorem{rem}[thm]{Remark}
\newtheorem{con}[thm]{Construction}

\newcommand{\ar}{\mathcal{R}}
\newcommand{\el}{\mathcal{L}}
\newcommand{\h}{\mathcal{H}}
\newcommand{\dee}{\mathcal{D}}
\newcommand{\jay}{\mathcal{J}}
\newcommand{\K}{\mathcal{K}}

\begin{document}
\title{\large{Ascending chain conditions on right ideals of semigroups}}
\author{Craig Miller}
\address{Department of Mathematics, University of York, YO10 5DD, UK\\
craig.miller@york.ac.uk}
\maketitle
\vspace{-1em}
\begin{abstract}
We call a semigroup $S$ {\em right noetherian} if it satisfies the ascending chain condition on right ideals, and we say that $S$ {\em satisfies ACCPR} if it satisfies the ascending chain condition on {\em principal} right ideals.  We investigate the behavior of these two conditions with respect to ideals and ideal extensions, with a particular focus on minimal and 0-minimal one-sided ideals.  In particular, we show that the property of satisfying ACCPR is inherited by right and left ideals.  On the other hand, we exhibit an example of a right noetherian semigroup with a minimal ideal that is not right noetherian.
\end{abstract}
~\\
\textit{Keywords}: Semigroup, one-sided ideal, ascending chain condition, semigroup act.\\
\textit{Mathematics Subject Classification 2010}: 20M10, 20M12.
\vspace{1em}

\section{Introduction\nopunct}
\label{sec:intro}

A {\em finiteness condition} for a class of universal algebras is a property that is satisfied by at least all finite members of that class.  Ascending chain conditions are classic examples.  
A poset $P$ satisfies the {\em ascending chain condition} if every ascending chain $$a_1\leq a_2\leq\cdots$$ eventually stabilises.  
Ascending chain conditions on ideals of rings, introduced by Noether in the landmark paper \cite{Noether:1921}, have played a crucial role in the development of ring theory, appearing in major results such as Hilbert's basis theorem, Krull's height theorem and the Noether-Lasker theorem.  Analagous conditions naturally arise in semigroup theory.  A {\em right ideal} of a semigroup $S$ is a subset $I\subseteq S$ such that $IS\subseteq I.$  We call $S$ {\em right noetherian} if its poset of right ideals (under containment) satisfies the ascending chain condition, and we say that $S$ {\em satisfies ACCPR} if its poset of principal right ideals satisfies the ascending chain condition.  Right noetherian semigroups have received a significant amount of attention; see for instance \cite{Aubert:1964, Davvaz:2020, Jespers:1999, Miller:2021, Sat:1972}.  Semigroups satisfying ACCPR have been considered in \cite{Liu:2004, Mazurek:2009, Stopar:2012}.

A related semigroup finiteness condition arises from the notion of a {\em right congruence}, that is, an equivalence relation $\rho$ on a semigroup $S$ such that $(a, b)\in\rho$ implies $(ac, bc)\in\rho$ for all $a, b, c\in S$.  We call a semigroup {\em strongly right noetherian} if its poset of right congruences satisfies the ascending chain condition\footnote{Strongly right noetherian semigroups are also known in the literature as `right noetherian', and the term `weakly right noetherian' has been used to denote semigroups that satisfy the ascending chain condition on right ideals.}.  The study of such semigroups was initiated by Hotzel in \cite{Hotzel:1975}, and further developed in \cite{Kozhukhov:1980, Kozhukhov:2003, Miller:2020}.  As the name suggests, strongly right noetherian semigroups are right noetherian \cite[Lemma 2.7]{Miller:2020}.  The converse, however, does not hold.  Indeed, unlike the situation for rings, the lattice of right ideals of a semigroup is not in general isomorphic to the lattice of right congruences.  For example, the lattice of right ideals of a group is trivial, whereas its lattice of right congruences is isomorphic to its lattice of subgroups.  Consequently, every group is trivially right noetherian, but a group is strongly right noetherian if and only if it satisfies the ascending chain condition on subgroups.

For any finiteness condition, it is natural to investigate the behavior of the condition with respect to substructures, quotients and extensions.  In particular, for a semigroup finiteness condition $\mathcal{P},$ the following questions arise.  For a semigroup $S$ and an ideal $I$ of $S$:
\begin{enumerate}
\item does $I$ satisfy $\mathcal{P}$ if $S$ satisfies $\mathcal{P}$?
\item does the Rees quotient $S/I$ satisfy $\mathcal{P}$ if $S$ satisfies $\mathcal{P}$?
\item does $S$ satisfy $\mathcal{P}$ if both $I$ and $S/I$ satisfy $\mathcal{P}$?
\end{enumerate}
The purpose of this article is to study the finiteness conditions of satisfying ACCPR and of being right noetherian, with (1)-(3) as our guiding questions.

The paper is organised as follows.  In Section \ref{sec:prelim} we provide the necessary preliminary material.  In particular, we collect some known results regarding the properties of satisfying ACCPR and being right noetherian, including some equivalent formulations of these conditions.  The main results of the paper are contained in Sections \ref{sec:ACCPR} and \ref{sec:rn}.  In Section \ref{sec:ACCPR} we consider the property of satisfying ACCPR, while Section \ref{sec:rn} is concerned with the property of being right noetherian.  These two sections follow the same format: they split into two subsections, the first concerning ideals in general and the second focusing on (0-)minimal ideals.

\section{Preliminaries\nopunct}
\label{sec:prelim}

\subsection{Ideals and related concepts}

Let $S$ be a semigroup.  We denote by $S^1$ the monoid obtained from $S$ by adjoining an identity if necessary (if $S$ is already a monoid, then $S^1=S$).  Similarly, we denote by $S^0$ the monoid obtained from $S$ by adjoining a zero if necessary.
%We denote the set of idempotents of $S$ by $E(S).$  If $S=E(S),$ it is called a {\em band}.  A {\em semilattice} is a commutative band. 
%An element $a\in S$ is said to be {\em regular} if there exists $b\in S$ such that $a=aba.$  The semigroup $S$ is said to be {\em regular} if every element of $S$ is regular.  It turns that for every regular element $a\in S$ there exists $b\in S$ such that $a=aba$ and $b=bab$; in this case, the element $b$ is said to be an {\em inverse} of $a,$ and vice versa.  If $S$ is regular and each of its elements has a unique inverse, then $S$ is called {\em inverse}.  If $S$ is inverse, then its set of idempotents $E(S)$ forms a semilattice.

Recall that a right ideal of $S$ is a subset $I\subseteq S$ such that $IS\subseteq I,$ and the principal right ideals of $S$ are those of the form $aS^1,$ $a\in S.$  Dually, we have the notion of ({\em principal}) {\em left ideals}.  An {\em ideal} of $S$ is a set that is both a right ideal and left ideal of $S,$ and the {\em principal ideals} of $S$ are the sets $S^1aS^1,$ $a\in S$.
Principal (one-sided) ideals determine the five Green's relations on a semigroup.  Green's relation $\ar$ on $S$ is given by
$$a\,\ar\,b\Leftrightarrow aS^1=bS^1.$$
Green's relations $\el$ and $\jay$ are defined similarly, in terms of principal left ideals and principal ideals, respectively.  Green's relation $\h$ is defined as $\h=\ar\cap\el,$ and finally we have $\dee=\ar\circ\el\,(=\el\circ\ar=\el\vee\ar)$.
It is clear from the definitions that Green's relations are equivalence relations on $S.$  Moreover, $\ar$ is a right congruence on $S$ and $\el$ is a left congruence on $S.$ 
%Associated to Green's relation $\ar$ is the preorder $\leq_{\ar}$ on $S$ given by $$a\leq_{\ar}b\Leftrightarrow aS^1\subseteq bS^1,$$ which induces a partial order on the set of $\ar$-classes.  We note that this poset is isomorphic to the poset of principal right ideals of $S$ via the isomorphism $R_a\mapsto aS^1,$ where $R_a$ denotes the $\ar$-class of $a.$ 

It is easy to see that the following inclusions between Green's relations hold: 
$$\h\subseteq\el,\, \h\subseteq\ar,\, \el\subseteq\dee,\, \ar\subseteq\dee,\, \dee\subseteq\jay.$$
It can be easily shown that every right (resp.\ left, two-sided) ideal is a union of $\ar$-classes (resp.\ $\el$-classes, $\jay$-classes).
A semigroup with no proper right (resp.\ left) ideals is called {\em right} (resp.\ {\em left}) {\em simple}.  A semigroup is {\em simple} if it has no proper ideals.  Clearly any right/left simple semigroup is simple.

A right (resp.\ left, two-sided) ideal $I$ of $S$ is said to be {\em minimal} if there is no right (resp.\ left, two-sided) ideal of $S$ properly contained in $I.$ 
It turns out that, considered as semigroups, minimal right (resp.\ left) ideals are right (resp.\ left) simple \cite[Theorem 2.4]{Clifford:1948}, and minimal ideals are simple \cite[Theorem 1.1]{Clifford:1948}.  A semigroup contains at most one minimal ideal but may possess multiple minimal right/left ideals.  The minimal ideal of a semigroup $S$ is also known as the {\em kernel}, and we denote it by $\K(S).$  If $S$ has a minimal right (resp.\ left) ideal, then $\K(S)$ exists and is equal to the union of all the minimal right (resp.\ left) ideals \cite[Theorem 2.1]{Clifford:1948}.
A {\em completely simple} semigroup is a simple semigroup that possesses both minimal right ideals and minimal left ideals.  A semigroup has both minimal right ideals and minimal left ideals if and only if it has a completely simple kernel \cite[Theorem 3.2]{Clifford:1948}.

For semigroups with zero, the theory of minimal ideals becomes trivial, so we require the notion of 0-minimality.  Suppose that $S$ has a zero element 0.  For convenience, we will usually just write the set $\{0\}$ as 0.  We say that $S$ is {\em right} (resp.\ {\em left}) {\em 0-simple} if $S^2\neq 0$ and $S$ contains no proper right (resp.\ left) ideals except 0, and $S$ is called {\em 0-simple} if $S^2\neq 0$ and 0 is its only proper ideal. 
A right (resp.\ left, two-sided) ideal $I$ of $S$ is said to be {\em 0-minimal} if 0 is the only proper right (resp.\ left, two-sided) ideal of $S$ contained in $I.$  A 0-minimal ideal $I$ of $S$ is either null or 0-simple \cite[Theorem 2.29]{Clifford:1961} (a semigroup $T$ is {\em null} if $T^2=0$).  If $I$ is a 0-minimal ideal of $S$ containing a 0-minimal right ideal of $S,$ then $I$ is the union of all the 0-minimal right ideals of $S$ contained in $I$ \cite[Theorem 2.33]{Clifford:1961}.  A {\em completely 0-simple} semigroup is a 0-simple semigroup that possesses both 0-minimal right ideals and 0-minimal left ideals.

For any 0-minimal right ideal $R$ of $S,$ since $R^2$ is a right ideal of $S$ contained in $R,$ it follows by 0-minimality that either $R$ is null or $R^2=R.$  Similarly, for any $a\in R$ we have either $aR=0$ or $aR=R.$  If $R$ is a 0-minimal right ideal such that $R^2=R,$ we say that $R$ is {\em globally idempotent}.
In contrast to the situation for 0-minimal two-sided ideals, globally idempotent 0-minimal right ideals need not be right 0-simple; see the remark immediately after Lemma 2.31 in \cite{Clifford:1961}.  

Let $R$ be a globally idempotent 0-minimal right ideal of $S.$  For any $s\in S,$ the set $sR$ is either 0 or a 0-minimal right ideal of $S$ \cite[Lemma 2.32]{Clifford:1961}.  Thus, the set $SR,$ the (two-sided) ideal of $S$ generated by $R,$ is a union of 0-minimal right ideals of $S.$  Let $A^R$ denote the union of $\{0\}$ and all the null 0-minimal right ideals of $S$ contained in $SR,$ and let $B^R$ denote the union of all the globally idempotent 0-minimal right ideals of $S$ contained in $SR.$  We call $A^R$ the {\em null part} of $SR,$ and $B^R$ the {\em globally idempotent part} of $SR.$  We note that $A^R$ may equal 0.  We provide a structure theorem describing $SR$ in terms of $A^R$ and $B^R$; in order to do so, we first recall a couple of definitions.\par
Let $S$ be a semigroup with 0 that is the union of subsemigroups $S_i$ $(i\in I).$  If $S_i\cap S_j=0$ for all $i, j\in I,$ $i\neq j,$ we say that $S$ is the {\em 0-disjoint union} of $S_i$ ($i\in I$).  If, additionally, $S_iS_j=0$ for all $i, j\in I,$ $i\neq j,$ we say that $S$ is the {\em 0-direct union} of $S_i$ ($i\in I$).

\begin{thm}\cite[Theorem 6.19]{Clifford:1967}
\label{thm:SR}
Let $S$ be a semigroup with a globally idempotent 0-minimal right ideal $R.$  Then:
\begin{enumerate}
\item $SR$ is a 0-disjoint union of $A^R$ and $B^R$;
\item $A^R$ is a null semigroup and an ideal of $S$;
\item $B^R$ is a 0-simple semigroup and a right ideal of $S$;
\item a subset of $B^R$ is a (0-minimal) right ideal of $B^R$ if and only if it is a (0-minimal) right ideal of $S.$
\end{enumerate}
\end{thm}

Let $S$ be a semigroup with 0.  The {\em right socle} of $S$ is the union of 0 and all the 0-minimal right ideals of $S.$  We denote the right socle by $\Sigma^r(S),$ or just $\Sigma^r$ when there is no ambiguity.  It turns out that $\Sigma^r$ is a (two-sided) ideal of $S$ \cite[Theorem 6.22]{Clifford:1967}.

Let $A^r=A^r(S)$ denote the union of 0 and all the null 0-minimal right ideals of $S,$ and let $B^r=B^r(S)$ denote the union of 0 and all the globally idempotent 0-minimal right ideals of $S.$  We call $A^r$ the {\em null part} of $\Sigma^r,$ and $B^r$ the {\em globally idempotent part} of $\Sigma^r.$  Of course, if $S$ has no 0-minimal right ideals then $\Sigma^r=A^r=B^r=0.$ 

\begin{thm}\cite[Theorem 6.23]{Clifford:1967}
\label{thm:rightsocle}
Let $S$ be a semigroup with 0.  Then:
\begin{enumerate}
\item $\Sigma^r$ is a 0-disjoint union of $A^r$ and $B^r$;
\item $A^r$ is a null semigroup and an ideal of $S$;
\item $B^r$ is a right ideal of $S$;
\item either $B^r=0$ or there exists a set $\{R_i : i\in I\}$ of globally idempotent 0-minimal right ideals of $S$ such that $B^r$ is the 0-direct union of the 0-simple semigroups $B^{R_i}$ $(i\in I).$
\end{enumerate}
\end{thm} 

The above definitions and results regarding 0-minimal right ideals have obvious duals for 0-minimal left ideals, and we use analogous notation ($A^L,$ $B^L,$ $\Sigma^l(S),$ etc.).

Given an ideal $I$ of $S,$ the {\em Rees quotient} of $S$ by $I,$ denoted by $S/I,$  is the set $(S\!\setminus\!I)\cup\{0\}$ with multiplication given by
$$a\cdot b=\begin{cases}
ab&\text{ if }a, b, ab\in S\!\setminus\!I,\\
0&\text{ otherwise.}
\end{cases}$$
Let $J$ be a $\jay$-class of $S.$  The {\em principal factor} of $J$ is defined as follows.  If $J=\K(S)$ then its principal factor is itself.  Otherwise, the principal factor of $J$ is the Rees quotient of the principal ideal $S^1xS^1,$ where $x$ is any element of $J,$ by the ideal $(S^1xS^1)\!\setminus\!J.$
The {\em principal factors} of $S$ are the principal factors of its $\jay$-classes.  As mentioned above, if $\K(S)$ exists then it is simple; all other principal factors are either $0$-simple or null.

\subsection{Acts}

Semigroup acts play the analagous role in semigroup theory as that of modules in the theory of rings.  We provide some basic definitions about acts; one should consult \cite{Kilp:2000} for more information.\par
A ({\em right}) {\em $S$-act} is a non-empty set $A$ together with a map 
$$A\times S\to A, (a, s) \mapsto as$$
such that $a(st)=(as)t$ for all $a\in A$ and $s, t\in S.$  A subset $B$ of an $S$-act $A$ is a {\em subact} of $A$ if $bs\in B$ for all $b\in B$ and $s\in S.$
Note that $S$ itself is an $S$-act via right multiplication, and its subacts are precisely its right ideals.  For clarity, a right ideal $I$ of $S$ will be written as $I_S$ when we are viewing it as a subact (including the case $I=S$).\par
Given an $S$-act $A$ and a subact $B$ of $A,$ the {\em Rees quotient} of $A$ by $B,$ denoted by $A/B,$ is the $S$-act with universe $(A\!\setminus\!B)\cup\{0\}$ and action given by: for all $a\in A/B$ and $s\in S,$
$$a\cdot s=
\begin{cases}
as &\text{if }a, as\in A\!\setminus\!B,\\
0 &\text{otherwise.}
\end{cases}$$
A subset $X$ of an $S$-act $A$ is a {\em generating set} for $A$ if $A=XS^1,$ and $A$ is said to be {\em finitely generated} (resp.\ {\em principal}) if it has a finite (resp.\ one-element) generating set.  Thus, the principal right ideals of $S$ are precisely the principal subacts of $S_S$.

Note that when we speak of a right ideal $I$ of a semigroup $S$ being generated by a set $X,$ we mean that $X$ generates $I$ as an $S$-act, i.e.\ $I=XS^1.$  

We call an $S$-act $A$ is {\em noetherian} if the poset of subacts of $A$ (under containment) satisfies the ascending chain condition, and we say that $A$ {\em satisfies ACCP} if the poset of principal subacts satisfies the ascending chain condition.  In particular, the $S$-act $S_S$ is noetherian (resp.\ satisfies ACCP) if and only if $S$ is right noetherian (resp.\ satisfies ACCPR).

Given an $S$-act $A,$ we define an equivalence relation $\ar_S$ on $A$ by
$$a\,\ar_S\,b\Leftrightarrow aS^1=bS^1.$$
Notice that $\ar_S$ on the $S$-act $S_S$ coincides with Green's relation $\ar$ on $S.$  We denote the $\ar_S$-class of an element $a\in A$ by $R_a$.  There is a natural partial order $\leq$ on the set of $\ar_S$-classes of $A$ given by 
$$R_a\leq R_b\Leftrightarrow aS^1\subseteq bS^1.$$  
It is easy to see that the poset of $\ar_S$-classes is isomorphic to the poset of principal subacts of $A$ via the isomorphism $R_a\to aS^1.$

We call an $S$-act $A$ {\em simple} if it contains no proper subact.  If an $S$-act $A$ has a {\em zero} 0 (that is, $0s=0$ for all $s\in S$), we say that $A$ is {\em 0-simple} if $\{0\}$ is its only proper subact.  Notice that the simple subacts of $S_S$ are precisely the minimal right ideals of $S,$ and, if $S$ has a zero 0, the 0-simple subacts of $S_S$ are precisely the 0-minimal right ideals of $S.$

\subsection{Foundational results}

In this subsection we establish some foundational results, many of which will be required later in the paper.  Some of these results are folklore but we provide proofs for completeness.  We begin by presenting some equivalent characterisations of the property of satisfying ACCP.

\begin{prop}
\label{prop:maxcondition}
Let $S$ be a semigroup and let $A$ be an $S$-act.  Then the following are equivalent:
\begin{enumerate}
\item $A$ satisfies ACCP;
\item the poset of $\ar_S$-classes of $A$ satisfies the ascending chain condition;
\item every non-empty set of principal subacts of $A$ contains a maximal element.
\end{enumerate}
\end{prop}

\begin{proof}
(1)$\Leftrightarrow$(2) follows from the fact, established above, that the poset of $\ar_S$-classes of $A$ is isomorphic to the poset of principal subacts of $A.$

(1)$\Rightarrow$(3).  Suppose for a contradiction that there exists a non-empty set $\mathcal{F}$ of principal subacts of $A$ with no maximal element.  Pick any $a_1S^1\in\mathcal{F}.$  Since $a_1S^1$ is not maximal, there exists $a_2S^1\in\mathcal{F}$ such that $a_1S^1\subsetneq a_2S^1.$  Continuing in this way, we obtain an infinite ascending chain
$$a_1S^1\subsetneq a_2S^1\subsetneq\cdots$$
of principal subacts of $A,$ contradicting the fact that $A$ satisfies ACCP.

(3)$\Rightarrow$(1).  Consider an ascending chain 
$$a_1S^1\subseteq a_2S^1\subseteq\cdots$$
where $a_i\in A.$  By assumption, the set $\{a_iS^1 : i\in\mathbb{N}\}$ contains a maximal element, say $a_mS^1.$  Then we must have that $a_nS^1=a_mS^1$ for all $n\geq m.$  Thus $A$ satisfies ACCP.
\end{proof}

\begin{cor}
\label{cor:maxcondition}
The following are equivalent for a semigroup $S$:
\begin{enumerate}
\item $S$ satisfies ACCPR;
\item every non-empty set of principal right ideals of $S$ contains a maximal element;
\item the poset of $\ar$-classes of $S$ satisfies the ascending chain condition.
\end{enumerate}
\end{cor}

We now provide several equivalent formulations of the property of being noetherian for acts.  For this result, recall that an {\em antichain} of a poset is a subset consisting of pairwise incomparable elements.

\begin{thm}
\label{thm:principal}
Let $S$ be a semigroup and let $A$ be an $S$-act.  Then the following are equivalent:
\begin{enumerate}
\item $A$ is noetherian;
\item every subact of $A$ is finitely generated;
\item every non-empty set of subacts of $A$ contains a maximal element;
\item $A$ satisfies ACCP and contains no infinite antichain of principal subacts;
\item the poset of $\ar_S$-classes of $A$ satisfies the ascending chain condition and contains no infinite antichain.
\end{enumerate}
\end{thm}

\begin{proof}
The proof that (1), (2) and (3) are equivalent is essentially the same as that of the analogue for modules over rings; see \cite[Section 10.1]{Lang}.  $(4)\Leftrightarrow(5)$ follows from the fact that the poset of $\ar_S$-classes of $A$ is isomorphic to the poset of principal subacts of $A.$

$(1)\Rightarrow(4).$  Clearly $A$ satisfies ACCP.  Suppose for a contradiction that there exists an infinite antichain $\{a_iS^1 : i\in\mathbb{N}\}$ of principal subacts of $A.$  For each $n\in\mathbb{N},$ let $A_n$ be the subact $\{a_1, \dots, a_n\}S^1.$  Clearly $A_n\subseteq A_{n+1}$.  We cannot have $A_n=A_{n+1}$, for otherwise we would have $a_{n+1}\in a_iS^1$ for some $i\leq n,$ and hence $a_{n+1}S^1\subseteq a_iS^1,$ contradicting the fact that $a_iS^1$ and $a_nS^1$ are incomparable. 
Thus, we have an infinite strictly ascending chain $$A_1\subsetneq A_2\subsetneq\dots$$ of right ideals of $S,$ contradicting the assumption that $A$ is noetherian.

$(4)\Rightarrow(1).$  Suppose that $A$ is not noetherian but does satisfy ACCP.  We need to construct an infinite antichain of principal subacts of $A.$
Since $A$ is not noetherian, there exists an infinite strictly ascending chain
$$A_1\subsetneq A_2\subsetneq\cdots$$
of subacts of $A.$  Choose elements $a_1\in A_1$ and $a_k\in A_k\!\setminus\!A_{k-1}$ for $k\geq 2.$  Then certainly $a_kS^1$ is not contained in any $a_jS^1, j<k,$ since $a_jS^1\subseteq A_j$ and $a_k\in A_k\!\setminus\!A_j.$

Consider the infinite set $P_1=\{a_iS^1 : i\in\mathbb{N}\}$ of principal subacts of $A.$  Since $A$ satisfies ACCP, $P_1$ contains a maximal element, say $a_{k_1}S^1,$ by Proposition \ref{prop:maxcondition}.
Now consider the infinite set $P_2=\{a_iS^1 : i\geq k_1+1\}.$  Again, $P_2$ contains a maximal element, say $a_{k_2}S^1.$  Then $a_{k_2}S^1$ is not contained in $a_{k_1}S^1$ since $k_1<k_2,$ and $a_{k_1}S^1$ is not contained in $a_{k_2}S^1$ since $a_{k_1}$ is maximal in $P_1$. 
Similarly, the infinite set $P_3=\{a_iS^1 : i\geq k_2+1\}$ contains a maximal element, say $a_{k_3}S^1,$ and $a_{k_1}S^1, a_{k_2}S^1$ and $a_{k_3}S^1$ are pairwise incomparable.
Continuing this process ad infinitum, we obtain an infinite antichain $\{a_{k_i}S^1 : i\in\mathbb{N}\}$ of principal subacts of $A,$ as required.
\end{proof}

From Theorem \ref{thm:principal} we deduce a number of corollaries.

\begin{cor}\cite[Proposition 3.1 and Theorem 3.2]{Miller:2021}
\label{cor:principal}
The following are equivalent for a semigroup $S$:
\begin{enumerate}
\item $S$ is right noetherian;
\item every right ideal of $S$ is finitely generated;
\item every non-empty set of right ideals of $S$ contains a maximal element;
\item $S$ satisfies ACCPR and contains no infinite antichain of principal right ideals;
\item the poset of $\ar$-classes of $S$ satisfies the ascending chain condition and contains no infinite antichain.
\end{enumerate}
\end{cor}

\begin{cor}
\label{cor:R_S-classes}
Let $S$ be a semigroup.  Any $S$-act $A$ with finitely many $\ar_S$-classes is noetherian.
\end{cor}

\begin{cor}
\label{cor:R-classes}
Any semigroup with finitely many $\ar$-classes is right noetherian.
\end{cor}

\begin{cor}
\label{cor:simplesubacts}
Let $S$ be a semigroup, and let $A$ be an $S$-act (with 0) that is the union of (0-)simple subacts $A_i, i\in I.$  Then $A$ satisfies ACCP.  Furthermore, $A$ is noetherian if and only if $I$ is finite.
\end{cor}

\begin{proof}
It is clear that $A$ satisfies ACCPR.  The (0-)simple subacts of $A$ are clearly principal and form an antichain (under containment), so the second statement follows from Corollary \ref{cor:principal}.
\end{proof}

\begin{cor}
\label{cor:minrightideals}
Let $S$ be a semigroup (with 0) that is the union of (0-)minimal right ideals $R_i, i\in I,$ of $S$.  Then $S$ satisfies ACCPR.  Furthermore, $S$ is right noetherian if and only if $I$ is finite.
\end{cor}

%The next result can be proved in essentially the same way as that of the analogue for modules; see \cite[Chapter 10, Propositions 1.1 and 1.2]{Lang}.
The next result states, for both the properties of being noetherian and satisfying ACCP, an act has the property if and only if both a subact and the associated Rees quotient do.

\begin{prop}
\label{prop:acts}
Let $S$ be a semigroup, let $A$ be an $S$-act, and let $B$ be a subact of $A.$  Then $A$ is noetherian (resp.\ satisfies ACCP) if and only if both $B$ and $A/B$ are noetherian (resp.\ satisfy ACCP).
\end{prop}

\begin{proof}
Suppose that $A$ is noetherian (resp.\ satisfies ACCP).  Since any ascending chain of (principal) subacts of $B$ is also an ascending chain of (principal) subacts of $A,$ it follows that $B$ is noetherian (resp.\ satisfies ACCP).  Now consider an ascending chain $$C_1\subseteq C_2\subseteq\cdots$$ of (principal) subacts of $A/B.$ Let $\theta : A\to A/B$ be the quotient map, and set $D_n=C_n\theta^{-1}$ for all $n\in\mathbb{N}.$  Then we have an ascending chain $$D_1\subseteq D_2\subseteq\cdots$$ of (principal) subacts of $A.$  Since $A$ is noetherian, there exists $m\in\mathbb{N}$ such that $D_n=D_m$ for all $n\geq m.$  Then $C_n=D_n\theta=D_m\theta=C_m$ for all $n\geq m.$  Hence $A/B$ is noetherian (resp.\ satisfies ACCP).

Conversely, suppose that both $B$ and $A/B$ are noetherian (resp.\ satisfy ACCP).  Consider an ascending chain $$A_1\subseteq A_2\subseteq\cdots$$ of (principal) subacts of $A.$  If $A_n\cap B=\emptyset$ for all $n\in\mathbb{N},$ then each $A_n$ is a subact of $A/B,$ and hence the above chain must eventually stabilise since $A/B$ is noetherian.  Assume then that there exists $i_0\in\mathbb{N}$ such that $A_{i_0}\cap B\neq\emptyset.$  Setting $B_n=A_n\cap B$ and $C_n=(A_n\cup B)/B$ for all $n\geq i_0,$ we obtain ascending chains $$B_{i_0}\subseteq B_{i_0+1}\subseteq\cdots\;\;\text{ and }\;\;C_{i_0}\subseteq C_{i_0+1}\subseteq\cdots$$ of $B$ and $A/B,$ respectively.  Since $B$ and $A/B$ are noetherian, these chains eventually stabilise, and thus there exists $m\geq i_0$ such that $B_n=B_m$ and $C_n=C_m$ for all $n\geq m.$  Then we have that $$A_n=(A_n\!\setminus\!B_n)\cup B_n=(C_n\!\setminus\!\{0\})\cup B_n=(C_m\!\setminus\!\{0\})\cup B_m=(A_m\!\setminus\!B_m)\cup B_m=B_m$$ for all $n\geq m.$  Hence $A$ is noetherian.
\end{proof}

%\begin{cor}\label{cor:ACCP}
%Let $S$ be a semigroup, and let $C=A\cup B$ be an $S$-act where $A$ and $B$ are subacts of $C.$  Then $C$ satisfies ACCP if and only if both $A$ and $B$ satisfy ACCP.
%\end{cor}

%\begin{proof}
%The forward implication follows immediately from Proposition \ref{prop:acts}.  For the converse, suppose that $A$ and $B$ satisfy ACCP.  If $A\cap B=\emptyset,$ then $A/B\cong B\cup\{0\}$ clearly satisfies ACCP since $B$ satisfies ACCP. If $A\cap B\neq\emptyset,$  then $A/B\cong B/(A\cap B)$ satisfies ACCP by Proposition \ref{prop:acts}, since $B$ satisfies ACCP.  Since both $A$ and $A/B$ satisfy ACCP, it follows from Proposition \ref{prop:acts} that $C$ satisfies ACCP.
%\end{proof}

%The following proposition and corollary can be proved in a similar way as Proposition \ref{prop:acts} and Corollary \ref{cor:ACCP}.

%\begin{prop}\label{prop:wn}
%Let $S$ be a semigroup, let $A$ be an $S$-act, and let $B$ be a subact of $A.$  Then $A$ is noetherian if and only if both $B$ and $A/B$ are noetherian.
%\end{prop}

%\begin{cor}\label{cor:wn}
%Let $S$ be a semigroup, and let $C=A\cup B$ be an $S$-act where $A$ and $B$ are subacts of $C.$  Then $C$ is noetherian if and only if both $A$ and $B$ are noetherian.
%\end{cor}

We now focus on the semigroup conditions of being right noetherian and of satisfying ACCPR.  Every free semigroup satisfies ACCPR, but a free semigroup is right noetherian if and only if it is monogenic:

\begin{prop}\cite[Proposition 3.5]{Miller:2021}
\label{prop:free}
Let $X$ be a non-empty set.  The free semigroup $X^{\ast}$ on $X$ satisfies ACCPR, but $X^{\ast}$ is right noetherian if and only if $|X|=1.$
\end{prop}

Since every free semigroup satisfies ACCPR, this property is certainly not closed under quotients.  On the other hand, the property of being right noetherian {\em is} closed under quotients:

\begin{lemma}
\cite[Lemma 4.1]{Miller:2021}
\label{lem:quotient}
Let $S$ be a semigroup and let $\rho$ be a congruence on $S.$  If $S$ is right noetherian, then so is $S/\rho.$
\end{lemma}

The property of being right noetherian is not in general inherited by ideals; see \cite[Remark 6.10]{Miller:2021}.  Going in the other direction, if both an ideal and the associated Rees quotient are right noetherian, then so is the ideal extension:

\begin{prop}
\cite[Corollary 4.5]{Miller:2021}
\label{prop:idealext,wrn}
Let $S$ be a semigroup and let $I$ be an ideal of $S.$  If both $I$ and $S/I$ are right noetherian, then so is $S.$
\end{prop}

%Let $S$ be a semigroup and $T$ a subsemigroup of $S.$  We say that $T$ is {\em $\ar$-preserving} (in $S$) if Green's $\ar$-preorder on $T$ is the restriction of Green's $\ar$-preorder on $S$ to $T$; that is, $$\leq_{\ar_T}~=~\leq_{\ar_S}\cap~(T\times T).$$  It was shown in \cite{Miller:2021} that the property of being right noetherian is inherited by $\ar$-preserving subsemigroups.  It is easy to see that this is also the case for the property of satisfying ACCPR.

%\begin{lemma}
%Let $S$ be a semigroup and let $T$ be an $\ar$-preserving subsemigroup of $S.$  If $S$ is right noetherian (resp.\ satisfies ACCPR), then $T$ is right noetherian (resp.\ satisfies ACCPR).
%\end{lemma}

Recall that an element $a$ of a semigroup $S$ is {\em regular} if there exists $b\in S$ such that $a=aba,$ and $S$ is {\em regular} if all its elements are regular.  The property of being right noetherian is inherited by regular subsemigroups:  

\begin{prop}\cite[Corollary 5.7]{Miller:2021}
\label{prop:reg,rn}
Let $S$ be a semigroup with a regular subsemigroup $T.$  If $S$ is right noetherian then so is $T.$
\end{prop}

The corresponding statement for the property of satisfying ACCPR also holds:
%It is well known that regular subsemigroups are $\ar$-preserving; see \cite[Section 5]{Miller:2021} for a proof.

\begin{prop}
\label{prop:reg,ACCPR}
Let $S$ be a semigroup with a regular subsemigroup $T.$  If $S$ satisfies ACCPR then so does $T.$
\end{prop}

\begin{proof}
Consider an ascending chain
$$a_1T_1\subseteq a_2T^1\subseteq\cdots$$
of principal right ideals of $T.$  Then clearly we have an an ascending chain
$$a_1S_1\subseteq a_2S^1\subseteq\cdots$$
of principal right ideals of $S.$  Since $S$ is right noetherian, there exists $m\in\mathbb{N}$ such that $a_nS^1=a_mS^1$ for all $n\geq m.$  Therefore, for any $n\geq m$ there exists $s_n\in S$ such that $a_n=a_ms_n.$  Since $T$ is regular, there exists $x\in T$ such that $a_m=a_mxa_m.$  Then we have that
$$a_n=a_mxa_ms_n=a_m(xa_n)\in a_mT,$$
and hence $a_nT^1=a_mT^1.$  Thus $T$ satisfies ACCPR.
\end{proof}

\section{Semigroups Satisfying ACCPR\nopunct}
\label{sec:ACCPR}

In this section we consider the relationship between semigroups and their (one-sided) ideals with respect to the property of satisfying ACCPR.  We first consider ideals in general, and we then focus on minimal and 0-minimal ideals.

\subsection{General ideals}

It turns out that, unlike the property of being right noetherian, the property of satisfying ACCPR {\em is} closed under ideals.  In fact, we show that this property is closed under the more general class of ($m, n$)-ideals, introduced by Lajos in \cite{Lajos:1963}.\par
Let $m, n\in\mathbb{N}.$  An ($m, n$)-{\em ideal} of a semigroup $S$ is a subsemigroup $A$ of $S$ such that $A^mSA^n\subseteq A.$  Notice that any one-sided ideal is an ($m, n$)-ideal.  (1,1)-ideals are also known as {\em bi-ideals}, which were introduced by Good and Hughes in \cite{Good:1952}.

\begin{thm}
\label{thm:(m,n)-ideal}
Let $S$ be a semigroup, and let $A$ be an {\em($m, n$)}-ideal of $S$ for some $m, n\in\mathbb{N}.$  If $S$ satisfies ACCPR, then so does $A.$
\end{thm}

\begin{proof}
Assume for a contradiction that there exists an infinite strictly ascending chain $$a_1A^1\subsetneq a_2A^1\subsetneq\cdots$$ of principal right ideals of $A.$  Then clearly we have an ascending chain 
$$a_1S^1\subseteq a_2S^1\subseteq\cdots$$
of principal right ideals of $S.$  Since $S$ satisfies ACCPR, there exists $N\in\mathbb{N}$ such that $a_NS^1=a_pS^1$ for all $p\geq N.$  
Now, we have $a_{N+m+j}\in a_{N+m+j+1}A$ for each $j\in\{1, \dots, n\},$ $a_{N+m+n+1}\in a_NS,$ and $a_{N+i}\in a_{N+i+1}A$ for each $i\in\{0, \dots, m-1\}.$  Thus, we have
\begin{align*}
a_{N+m+1}&\in a_{N+m+2}A\subseteq a_{N+m+3}A^2\subseteq\cdots\subseteq a_{N+m+n+1}A^n\subseteq a_NSA^n\\&\subseteq a_{N+1}ASA^n\subseteq a_{N+2}A^2SA^n\subseteq\cdots\subseteq a_{N+m}A^mSA^n\subseteq a_{N+m}A,
\end{align*}
where the final containment follows from the fact that $A$ is an ($m, n$)-ideal of $S.$  But then $a_{N+m}A^1=a_{N+m+1}A^1,$ contradicting the assumption.
\end{proof}

\begin{cor}
\label{cor:ideal,ACCPR}
Let $S$ be a semigroup and let $I$ be a right/left/two-sided ideal of $S.$  If $S$ satisfies ACCPR, then so does $I.$
\end{cor}

It was noted in Section \ref{sec:prelim} that the property of satisfying ACCPR is not closed under quotients.  However, we shall see that this property {\em is} closed under Rees quotients.  First note that, given an ideal $I$ of $S,$ we have both the semigroup Rees quotient $S/I$ and the $S$-act Rees quotient $S_S/I_S$ (with the same universe).

\begin{lemma}
\label{lem:rq}
Let $S$ be a semigroup and let $I$ be an ideal of $S.$  Then $S/I$ satisfies ACCPR if and only if $S_S/I_S$ satisfies ACCP.
\end{lemma}

\begin{proof}
Since $I$ is an ideal of $S,$ for any $a, b\in S\!\setminus\!I$ we have that $aS^1\subseteq bS^1$ if and only if $a(S/I)^1\subseteq b(S/I)^1.$  From this fact the result readily follows.
%If $T$ does not satisfy ACCPR, then there exists an infinite strictly ascending chain $$a_1T^1\subsetneq a_2T^1\subsetneq\cdots$$ of principal right ideals of $T.$  For each $i\in\mathbb{N},$ there exists $t_i\in T^1$ such that $a_i=a_{i+1}t_i.$  Since $0T^1=\{0\},$ we conclude that $a_i\in S\!\setminus\!I$ for all $i\geq 2.$  Then, for all $i\geq 2,$ we have $t_i\in(S\!\setminus\!I)^1,$ so $a_i\subseteq a_{i+1}S^1.$   We cannot have $a_{i+1}\in a_iS^1$ for any $i\geq 2.$  Indeed, if $a_{i+1}=a_is,$ then $s\in(S\!\setminus\!I)^1,$ and thus $a_iT^1=a_{i+1}T^1,$ which is a contradiction.  Hence, we have an ascending chain $$a_1S^1\subsetneq a_2S^1\subsetneq\cdots$$ of principal subacts of $A,$ so $A$ does not satisfy ACCP.
\end{proof}

\begin{cor}
\label{cor:idealandRQ}
Let $S$ be a semigroup and let $I$ be an ideal of $S.$  If $S$ satisfies ACCPR, then both $I$ and $S/I$ satisfy ACCPR.
\end{cor}

\begin{proof}
We have that $I$ satisfies ACCPR by Corollary \ref{cor:ideal,ACCPR}.  Since $S_S$ satisfies ACCP, the quotient $S_S/I_S$ satisfies ACCP by Proposition \ref{prop:acts}, and hence $S/I$ satisfies ACCPR by Lemma \ref{lem:rq}.
\end{proof}

\begin{cor}
\label{cor:idealext}
Let $S$ be a semigroup and let $I$ be an ideal of $S.$  Then $S$ satisfies ACCPR if and only if $S/I$ satisfies ACCPR and (the $S$-act) $I_S$ satisfies ACCP.
\end{cor}

\begin{proof}
If $S$ satisfies ACCPR, then $S/I$ satisfies ACCPR by Corollary \ref{cor:idealandRQ}.  Since $S_S$ satisfies ACCP, the subact $I_S$ satisfies ACCP by Proposition \ref{prop:acts}.  The converse follows from Proposition \ref{prop:acts} and Lemma \ref{lem:rq}.
\end{proof}

Recall that a principal factor of a semigroup is either the minimal ideal (if it exists) or the Rees quotient of a certain ideal by another ideal.  Thus Corollary \ref{cor:idealandRQ} yields:

\begin{cor}
If a semigroup $S$ satisfies ACCPR, then so do all its principal factors.
\end{cor}

We shall show that the converse of Corollary \ref{cor:idealandRQ} does not hold.  To this end, we introduce the following construction.

\begin{con}
\label{con}
Let $S$ be a semigroup and let $A$ be an $S$-act.  Let $\{x_a : a\in A\}$ be a set in one-to-one correspondence with $A$ and disjoint from $S,$ and let $0$ be an element disjoint from $S\cup\{x_a : a\in A\}.$  Define a multiplication on $U=S\cup\{x_a : a\in A\}\cup\{0\},$ extending that on $S,$ by
$$x_as=x_{as}\,\text{ and }\,sx_a=x_ax_b=u0=0u=0$$
for all $s\in S,$ $a, b\in A$ and $u\in U.$  With this multiplication, $U$ is a semigroup, and we denote it by $\mathcal{U}(S, A).$  Notice that $\{x_a : a\in A\}\cup\{0\}$ is a null semigroup and an ideal of $S.$
\end{con}

\begin{prop}
\label{prop:con,ACCPR}
Let $S$ be a semigroup, let $A$ be an $S$-act, and let $U=\mathcal{U}(S, A).$  Then $U$ satisfies ACCPR if and only if $S$ satisfies ACCPR and $A$ satisfies ACCP.
\end{prop}

\begin{proof}
Let $I=\{x_a : a\in A\}\cup\{0\}.$  By Corollary \ref{cor:idealext}, we have that $U$ satisfies ACCPR if and only if $U/I$ satisfies ACCPR and $I_U$ satisfies ACCP.  Clearly $U/I\cong S^0$ satisfies ACCPR if and only if $S$ satisfies ACCPR.  It is easy to show that, for any $a, b\in A,$ we have $0U^1=0\subsetneq x_aU^1,$ and $x_aU^1\subseteq x_bU^1$ if and only if $aS^1\in bS^1.$   Thus, the poset of principal subacts of $I_U$ has the form $P\cup\{0\},$ where $P$ is isomorphic to the poset of principal subacts of $A.$  It follows that $I_U$ satisfies ACCP if and only if $A$ satisfies ACCP.  This completes the proof.
\end{proof}

We now show that the converse of Corollary \ref{cor:idealandRQ} does not hold. 

Let $S$ be a semigroup that satisfies ACCPR with an $S$-act $A$ that does not satisfy ACCP.  (For example, we can take $A$ to be any {\em semigroup} that does not satisfy ACCPR and $S$ to be a free semigroup with a surjective homomorphism $\theta : S\to A.$  We turn $A$ into an $S$-act by defining $a\cdot s=a(s\theta)$ for all $a\in A$ and $s\in S.$  We have that $S$ satisfies ACCPR by Proposition \ref{prop:free}, and it is straightforward to show that $A$ does not satisfy ACCP.)  The semigroup $U=\mathcal{U}(S, A)$ does not satisfy ACCPR by Proposition \ref{prop:con,ACCPR}.  On the other hand, the ideal $I=\{x_a : a\in A\}\cup\{0\}$ certainly satisfies ACCPR (indeed, any null semigroup satisfies ACCPR by Corollary \ref{cor:minrightideals}), and the Rees quotient $U/I\cong S^0$ satisfies ACCPR since $S$ satisfies ACCPR.

\vspace{0.7em}
We now consider conditions on an ideal $I$ such that converse of Corollary \ref{cor:idealandRQ} {\em does} hold.\par
Given a semigroup $S,$ we say that an element $a\in S$ has a {\em local right identity} (in $S$) if there exists $s\in S$ such that $a=as$; i.e.\ $a\in aS.$  If $S$ is a monoid or a regular semigroup, then clearly every element has a local right identity.

\begin{prop}
\label{prop:lri}
Let $S$ be a semigroup, let $I$ be an ideal of $S,$ and suppose that every element of $I$ has a local right identity in $I.$  Then $S$ satisfies ACCPR if and only if both $I$ and $S/I$ satisfy ACCPR.
\end{prop}

\begin{proof}
We show that the $S$-act $I_S$ satisfies ACCP.  The result then follows from Corollary \ref{cor:idealext}.  So, consider an ascending chain $$a_1S^1\subseteq a_2S^1\subseteq\cdots$$ of principal subacts of $I_S$.  Then for each $n\in\mathbb{N},$ we have that 
$$a_n=a_{n+1}S^1\subseteq a_{n+1}IS^1\subseteq a_{n+1}I,$$
using the fact that $a_{n+1}$ has a local right identity in $I.$  Therefore, we have an ascending chain 
$$a_1I^1\subseteq a_2I^1\subseteq\cdots$$ of principal right ideals of $I.$  Since $I$ satisfies ACCPR, there exists $m\in\mathbb{N}$ such that $a_nI^1=a_mI^1$ for all $n\geq m.$  Thus $a_nS^1=a_mS^1$ for all $n\geq m.$
\end{proof}

\begin{prop}
Let $S$ be a semigroup, let $I$ be an ideal of $S,$ and suppose that there is no infinite antichain of principal right ideals of $I.$  Then the following are equivalent:
\begin{enumerate}
\item $S$ satisfies ACCPR;
\item both $I$ and $S/I$ satisfy ACCPR;
\item $I$ is right noetherian and $S/I$ satisfies ACCPR.
\end{enumerate}
\end{prop}

\begin{proof}
(1)$\Rightarrow$(2) is Corollary \ref{cor:idealandRQ}.

(2)$\Rightarrow$(3).  Since $I$ satisfies ACCPR and has no infinite antichain of principal right ideals, it is right noetherian by Corollary \ref{cor:principal}.

(3)$\Rightarrow$(1).  Assume for a contradiction that $S$ does not satisfy ACCPR.  Then there exists an infinite strictly ascending chain $$a_1S^1\subsetneq a_2S^1\subsetneq\cdots$$ of principal right ideals of $S.$  We cannot have $a_i\in S\!\setminus\!I$ for any $i\in\mathbb{N},$ for then we would have an infinite ascending chain
$$a_i(S/I)^1\subseteq a_{i+1}(S/I)^1\subseteq\cdots$$
of principal right ideals of $S/I$.  Thus $a_i\in I$ for all $i\in\mathbb{N}.$

Consider the set $\{a_iI^1 : i\in\mathbb{N}\}$ of principal right ideals of $I.$ 
By assumption, this set does not contain an infinite antichain.  Also, we cannot have $a_iI^1\subseteq a_jI^1$ for any $i>j,$ for then we would have $a_iS^1=a_jS^1.$  Thus, there exist $i_1, j_1\in\mathbb{N}$ with $i_1<j_1$ such that $a_{i_1}I^1\subsetneq a_{j_1}I^1.$  Hence $a_{i_1}\in a_{j_1}I.$\par
Now consider the set $\{a_iI^1 : i\geq j_1\}.$  By a similar argument as above, there exist $i_2, j_2\in\mathbb{N}$ with $j_1\leq i_2<j_2$ such that $a_{i_2}I^1\subsetneq a_{j_2}I^1.$  Now, we have 
$$a_{i_1}\in a_{j_1}I\subseteq(a_{i_2}S^1)I=a_{i_2}(S^1I)\subseteq a_{i_2}I,$$
and hence $a_{i_1}I^1\subsetneq a_{i_2}I^1.$
Continuing this process ad infinitum, we obtain an infinite strictly ascending chain 
$$a_{i_1}I^1\subsetneq a_{i_2}I^1\subsetneq a_{i_3}I^1\subsetneq\cdots$$
of principal right ideals of $I,$ contradicting the fact that $I$ is right noetherian.  Hence, $S$ satisfies ACCPR.
\end{proof}

\subsection{Minimal and 0-minimal ideals}

In the remainder of this section we focus on minimal and 0-minimal (one-sided) ideals.  Recall that the minimal ideal of a semigroup $S,$ if it exists, is denoted by $\K(S).$

\begin{prop}
\label{prop:min_right_ideal}
Let $S$ be a semigroup with at least one minimal right ideal, and let $\K=\K(S).$  Then $S$ satisfies ACCPR if and only if $S/\K$ satisfies ACCPR.
\end{prop}

\begin{proof}
Clearly $\K,$ being the union of all the minimal right ideals of $S,$ satisfies ACCPR by Corollary \ref{cor:minrightideals}.
Consider $a\in\K.$  Then $a\in R$ for some minimal right ideal $R$ of $S.$  Clearly $a\K$ is a right ideal of $S$ contained in $R,$ so $a\K=R$ by the minimality of $R,$ and hence $a\in a\K.$  Thus every element of $\K$ has a local right identity.  The result now follows from Proposition \ref{prop:lri}.
\end{proof}

We now consider semigroups satisfying ACCPR with minimal left ideals. 

\begin{thm}
\label{thm:min_left_ideal}
Let $S$ be a semigroup that satisfies ACCPR.  Then $S$ has a minimal left ideal if and only if $S$ has a completely simple kernel.
\end{thm}

\begin{proof}
If $S$ has a completely simple kernel, then, as established in Section \ref{sec:prelim}, $S$ has minimal left ideals.

Now suppose that $S$ has a minimal left ideal.  Then $\K=\K(S)$ is the union of all the minimal left ideals of $S.$  We shall prove that $\K$ has an idempotent, and then $\K$ is completely simple by \cite[Theorem 8.14]{Clifford:1967}.

Let $L$ be a minimal left ideal of $S,$ and consider the set $\{aS^1 : a\in L\}$ of principal right ideals of $S.$  This set contains a maximal element, say $xS^1.$  Since $Lx$ is left ideal of $S$ contained in $L,$ we have that $L=Lx$ by the minimality of $L.$  Thus $x=yx$ for some $y\in L,$ and hence $xS^1\subseteq yS^1.$  Since $xS^1$ is maximal in the set $\{aS^1 : a\in L\},$ we conclude that $xS^1=yS^1.$  Then $y=xs$ for some $s\in S^1,$ and hence $$y^2=y(xs)=(yx)s=xs=y.$$
Thus $L\subseteq\K$ has an idempotent, as required.
\end{proof}

\begin{cor}
Let $S$ be a semigroup.  Then $S$ satisfies ACCPR and has a minimal left ideal if and only if $S$ has a completely simple minimal ideal $\K$ and $S/\K$ satisfies ACCPR.
\end{cor}

\begin{proof}
The forward implication follows from Theorem \ref{thm:min_left_ideal} and Corollary \ref{cor:idealext}.  Conversely, since $S$ has a completely simple minimal ideal, it certainly has a minimal left ideal, and $S$ satisfies ACCPR by Proposition \ref{prop:lri}, since every element of $\K$ has a local right identity.
\end{proof}  

The following result is an analogue of Theorem \ref{thm:min_left_ideal} for 0-minimal ideals.

\begin{thm}
\label{thm:0-simple}
Let $S=S^0$ be a semigroup that satisfies ACCPR and has a 0-minimal ideal $I.$  Then $I$ contains a globally idempotent 0-minimal left ideal of $S$ if and only if $I$ is completely 0-simple.
\end{thm}

\begin{proof}
($\Rightarrow$)  Suppose that $I$ contains a globally idempotent 0-minimal left ideal $L$ of $S.$  Since $L^2=L,$ therefore $I^2\neq 0$ and hence $I$ is 0-simple.  We shall prove that $I$ contains an idempotent, and then it is completely 0-simple by \cite[Theorem 8.22]{Clifford:1967}.

Recall that for any $a\in L,$ either $La=L$ or $La=0.$  Consider the set $$P=\{aS^1 : a\in L, La=L\}$$ of principal right ideals of $S.$  By the 0-minimality of $L,$ we have $L=S^1a$ for each $a\in L.$  Since $L=L^2,$ there exist $b, c\in L$ such that $bc\in L,$ and hence $L=S^1(bc)=(S^1b)c=Lc.$  Thus $P$ is non-empty.  Since $S$ satisfies ACCPR, $P$ contains a maximal element, say $xS^1.$  Then $x\in L$ and $L=Lx.$  Thus $x=yx$ for some $y\in L,$ and hence $xS^1\subseteq yS^1.$  Since $(Ly)x=L(yx)=Lx=L,$ we cannot have $Ly\neq 0,$ so $Ly=L$ and hence $yS^1\in P.$  Since $xS^1$ is maximal in $P,$ we conclude that $xS^1=yS^1.$  Then, as in the proof of Theorem \ref{thm:min_left_ideal}, we have $y^2=y,$ so $I$ contains an idempotent, as required.

($\Leftarrow$)  If $I$ is completely 0-simple, then it has a globally idempotent 0-minimal ideal $L.$  We have that 
$$SL=SL^2=(SL)L\subseteq IL\subseteq L,$$
and $SL\neq 0$ since $L^2=L,$ so $SL=L$ by the 0-minimality of $L.$  Thus $L$ is a left ideal of $S.$  Clearly any left ideal of $S$ contained in $L$ also a left ideal of $I,$ so it follows from the 0-minimality of $L$ in $I$ that $L$ is 0-minimal in $S.$
\end{proof}

\begin{cor}
\label{cor:0-simple}
Let $S$ be a 0-simple semigroup.  Then $S$ satisfies ACCPR and has a 0-minimal left ideal if and only if $S$ is completely 0-simple.
\end{cor}

\begin{proof}
Suppose that $S$ satisfies ACCPR and has a 0-minimal left ideal $L.$  Since $S$ is 0-simple, we have that $L^2\neq 0$ by \cite[Lemma 2.34]{Clifford:1961}, and hence $L$ must be globally idempotent.  It follows from Theorem \ref{thm:0-simple} that $S$ is completely 0-simple.

The converse clearly holds.
\end{proof}

\begin{cor}
\label{cor:0-min_left_ideal}
Let $S=S^0$ be a semigroup with a globally idempotent 0-minimal left ideal $L.$  If $S$ satisfies ACCPR, then the globally idempotent part $B^L$ of $LS$ is completely 0-simple.
\end{cor}

\begin{proof}
By the left-right dual of Theorem \ref{thm:SR}, $B^L$ is a left ideal of $S.$  Therefore, since $S$ satisfies ACCPR, $B^L$ satisfies ACCPR by Corollary \ref{cor:ideal,ACCPR}.  Also by the left-right dual of Theorem \ref{thm:SR}, $B^L$ is 0-simple and has globally idempotent 0-minimal left ideals (of itself).
Hence, by Corollary \ref{cor:0-simple}, $B^L$ is completely 0-simple.
\end{proof}

Recall that the left socle $\Sigma^l=\Sigma^l(S)$ of a semigroup $S$ with 0 is the 0-disjoint union of $A^l$ and $B^l,$ which are the null part and globally idempotent part of $\Sigma^l$, respectively.  Note that since $A^l$ is an ideal of $S,$ we may view it as a subact of $S_S$.

\begin{thm}
\label{thm:leftsocle,ACCPR}
Let $S=S^0$ be a semigroup, and let $\Sigma^l=\Sigma^l(S).$  Then the following are equivalent:
\begin{enumerate}
\item $S$ satisfies ACCPR;
\item $B^l$ is either 0 or the 0-direct union of completely 0-simple semigroups $B_i$ $(i\in I),$ the $S$-act $A^l_S$ satisfies ACCP, and $S/\Sigma^l$ satisfies ACCPR.
\end{enumerate}
\end{thm}

\begin{proof}
(1)$\Rightarrow$(2).  Suppose that $B^l\neq 0.$  Then, by the left-right dual of Theorem \ref{thm:rightsocle}, there exists a set $\{L_i : i\in I\}$ of globally idempotent 0-minimal left ideals of $S$ such that $B^l$ is the 0-direct union of the 0-simple semigroups $B_i=B^{L_i}$ $(i\in I).$  Each $B_i$ is completely 0-simple by Corollary \ref{cor:0-min_left_ideal}.  The subact $A^l_S$ of $S_S$ satisfies ACCP by Proposition \ref{prop:acts}, and $S/\Sigma^l$ satisfies ACCPR by Corollary \ref{cor:idealandRQ}.

(2)$\Rightarrow$(1).  Let $T$ denote the Rees quotient $S/A^l.$  Since $A^l_S$ satisfies ACCP, by Corollary \ref{cor:idealext} it suffices to prove that $T$ satisfies ACCPR.  Notice that $B_l$ is (isomorphic to) an ideal of $T.$  Since $B^l$ is either 0 or the 0-direct union of completely 0-simple semigroups, it satisfies ACCPR by Corollary \ref{cor:minrightideals}, and every element of $B^l$ has a local right identity.  Now, $T/B^l\cong S/\Sigma^l$ (by the Third Isomorphism Theorem), so $T/B_l$ satisfies ACCPR by assumption.  Hence, by Proposition \ref{prop:lri}, $T$ satisfies ACCPR, as required.
\end{proof}

If $S=S^0$ has no null 0-minimal ideals then $\Sigma^l(S)=B^l$, so by Theorem \ref{thm:leftsocle,ACCPR} we have:

\begin{cor}
Let $S=S^0$ be a semigroup without null 0-minimal ideals, and let $\Sigma^l=\Sigma^l(S).$  Then the following are equivalent:
\begin{enumerate}
\item $S$ satisfies ACCPR;
\item $\Sigma^l$ is either 0 or the 0-direct union of completely 0-simple semigroups, and $S/\Sigma^l$ satisfies ACCPR.
\end{enumerate}
\end{cor}

We shall find some necessary and sufficient conditions for a semigroup $S=\Sigma^l(S)$ to satisfy ACCPR, but first we provide the following lemma.

\begin{lemma}
\label{lem:S=left_socle}
Let $S=S^0$ be a semigroup such that $S=\Sigma^l=\Sigma^l(S),$ and let $\Sigma^r=\Sigma^r(S).$  Then the following statements hold.
\begin{enumerate}
\item If $A^l\neq0,$ then $\{a, 0\}$ is a 0-minimal right ideal of $S$ for each $a\in A^l\!\setminus\!\{0\}.$
\item $A^r=A^l,$ $B^r\subseteq B^l,$ and $B^r$ is an ideal of $S.$
\item $\Sigma^r$ is the 0-direct union of $A^r$ and $B^r.$
\item $S/\Sigma^r\cong B^l/B^r.$
\end{enumerate}
\end{lemma}

\begin{proof}
By the left-right dual of Theorem \ref{thm:rightsocle}, $B^l$ is either 0 or the 0-direct union of 0-simple semigroups $B_i$ ($i\in I$).  Consider $x\in A^r\cap B^l$.  Since either $x=0$ or $x$ belongs to a 0-simple semigroup, we have $x\in J_x^2$, where $J_x$ denotes the $\jay$-class of $x.$  We have $J_x\subseteq A^r$ since $A^r$ is an ideal, and hence $x\in(A^r)^2=0,$ so $x=0.$  Thus $A^r\cap B^l=0,$ and hence $A^r\subseteq A^l.$  Since $A^l$ is an ideal of $S$ and $B^l$ is a left ideal of $S,$ it follows that $A^lB^l=0.$  Since $S=A^l\cup B^l$ and $(A^l)^2=0,$ we conclude that $A^lS=0.$  Therefore, if $A^l\neq0$ then $\{a, 0\}$ is a 0-minimal right ideal for each $a\in A^l\!\setminus\!\{0\}.$  Thus $A^l\subseteq A^r,$ and hence $A^r=A^l.$  Then $0=A^r\cap B^r=A^l\cap B^r,$ so $B^r\subseteq B^l.$ 

If $B^r=0,$ then it is clear that $B^r$ is an ideal of $S$ and that statements (3) and (4) hold, so we may assume that $B^r\neq0.$  Let 
$$J=\bigl\{j\in I : B_j\cap B^r\neq 0\bigr\}.$$
Consider $b\in B_j\cap B^r,$ $b\neq 0.$  We have that $B_j\subseteq SbS$ since $B_j$ is 0-simple, and hence $B_j\subseteq\Sigma^r$ as $\Sigma^r$ is an ideal.  We must have that $B_j\subseteq B^r,$ for otherwise we would have $b\in A^r$ (using the fact $B_j$ is 0-simple and $A^r$ is an ideal).  It follows that $B^r$ is the 0-direct union of $B_j$, $j\in J.$  

Now, $B^r$ is a right ideal of $S$ by Theorem \ref{thm:rightsocle}, so to prove that it is an ideal, it suffices to show that it is a left ideal.  So, let $s\in S$ and $b\in B^r.$  If $s\in A^l,$ then $sb\in A^lS=0.$  Suppose that $s\in B^l.$  We have that $s\in B_i$ and $b\in B_j$ for some $i\in I,$ $j\in J.$  If $i=j,$ then $sb\in B_j\subseteq B^r.$  If $i\neq j,$ then $sb\in B_iB_j=0.$  Thus $B^r$ is an ideal of $S.$

Since $A^r$ and $B^r$ are both ideals of $S,$ and $A^r\cap B^r=0,$ it follows that $A^rB^r=B^rA^r=0.$  Thus $\Sigma^r$ is the 0-direct union of $A^r$ and $B^r.$  

Since $B^r\subseteq B^l$ and $B^r$ is an ideal of $S,$ it is certainly an ideal of $B^l.$  Observing that the universe of $S/\Sigma^r$ is $(B^l\!\setminus\!B^r)\cup\{0\},$ it is easy to see that $S/\Sigma^r\cong B^l/B^r.$
\end{proof}

\begin{thm}
\label{thm:S=left_socle}
Let $S=S^0$ be a semigroup such that $S=\Sigma^l=\Sigma^l(S),$ and let $\Sigma^r=\Sigma^r(S).$  Then the following are equivalent: 
\begin{enumerate}
\item $S$ satisfies ACCPR;
\item $B^l$ is either 0 or the 0-direct union of completely 0-simple semigroups;
\item $\Sigma^r$ is either a null semigroup or the 0-direct union of a null semigroup and completely 0-simple semigroups, and either $\Sigma^r=S$ or $S/\Sigma^r$ is the 0-direct union of completely 0-simple semigroups.
\end{enumerate}
\end{thm}

\begin{proof}
(1)$\Rightarrow$(2) follows immediately from Theorem \ref{thm:leftsocle,ACCPR}.

(2)$\Rightarrow$(3).  By Lemma \ref{lem:S=left_socle}, $A^r=A^l,$ $B^r\subseteq B^l,$ and $\Sigma^r$ is the 0-direct union of $A^r$ and $B^r.$  If $B^r=0,$ then $\Sigma^r=A^r$ is a null semigroup.  If $B^r\neq0,$ then $B^l\neq0,$ so $B^l$ is the 0-direct union of completely 0-simple semigroups $B_i$ ($i\in I$).  As in the proof of Lemma \ref{lem:S=left_socle}, there exists a set $J\subseteq I$ such that $B^r$ is the 0-direct union of $B_j$, $j\in J.$  By Lemma \ref{lem:S=left_socle}, we have that $S/\Sigma^r\cong B^l/B^r.$  Thus, if $\Sigma^r\neq S$ then $S/\Sigma^r$ is (isomorphic to) the 0-direct union of $B_i,$ $i\in J\!\setminus\!I.$

(3)$\Rightarrow$(1).  We have that $S/A^r\cong B^r$ is a 0-direct union of  completely 0-simple semigroups, and hence satisfies ACCPR by Corollary \ref{cor:minrightideals}.  Therefore, to prove that $S$ satisfies ACCPR, by Corollary \ref{cor:idealext} it suffices to show that $A^r_S$ is noetherian (as an $S$-act).  If $A^r_S=0$ then it is obviously noetherian.  Otherwise, by Lemma \ref{lem:S=left_socle}, we have that $A^r_S$ is the union of 0-simple subacts $\{a, 0\}$ ($a\in A^r_S$), and hence $A^r_S$ is noetherian by Corollary \ref{cor:simplesubacts}. 
\end{proof}

\section{Right Noetherian Semigroups\nopunct}
\label{sec:rn}

In this section we consider right noetherian semigroups.  Paralleling the previous section, this section splits into two parts, the first of which deals with ideals in general, and the section concerns minimal and 0-minimal ideals.

\subsection{General ideals}

As mentioned in Section \ref{sec:ACCPR}, unlike the property of satisfying ACCPR, the property of being right noetherian is not closed under ideals.  The following result provides a condition under which ideals, and more generally ($m, n$)-ideals, inherit the property of being right noetherian.  In what follows, a right ideal $I$ of a semigroup $A$ is {\em decomposable} (in $A$) if $I=IA.$

\begin{prop}
\label{prop:bi-ideal,wrn}
Let $S$ be a semigroup, let $A$ be an ($m, n$)-ideal of $S,$ and suppose that every right ideal of $A$ is decomposable in $A.$  If $S$ is right noetherian, then so is $A.$
\end{prop}

\begin{proof}
Let $I$ be a right ideal of $A.$  Then $I=IA$ by assumption.  This implies that $I=IA^m=IA^n.$  Since $S$ is right noetherian and $IS^1$ is a right ideal of $S,$ there exists a finite set $X\subseteq I$ such that $IS^1=XS^1.$  For each $x\in X$ choose $y_x\in I$ such that $x\in y_xA^m,$ and let $Y=\{y_x : x\in X\}.$  We claim that $I=YA.$  Clearly $YA\subseteq I.$  Now consider $a\in I.$  Then $a=bv$ for some $b\in I$ and $v\in A^n,$ and $b=xs$ for some $x\in X$ and $s\in S^1.$  Now, $x=y_xu$ for some $u\in A^m.$  Therefore, we have that
$$a=y_x(usv)\in Y(A^mS^1A^n)\subseteq YA,$$
using the fact that $A$ is an ($m, n$)-ideal of $S.$  Thus $I\subseteq YA,$ and hence $I=YA,$ as desired.
\end{proof}

\begin{cor}
\label{cor:leftideal,wrn}
Let $S$ be a semigroup, and suppose that $A$ is a left ideal of $S$ such that every element of $A$ is regular in $S.$  If $S$ is right noetherian, then so is $A.$
\end{cor}

\begin{proof}
Let $I$ be a right ideal of $A.$  For any $a\in I$ there exists $b\in S$ such that $a=aba.$  Since $A$ is a left ideal, we have $ba\in A,$ so $a\in IA.$  Thus $I=IA$ is decomposable.  Hence, by Proposition \ref{prop:bi-ideal,wrn}, $A$ is right noetherian.
\end{proof}

By Propositions \ref{prop:idealext,wrn} and \ref{prop:bi-ideal,wrn} we have:

\begin{cor}
Let $S$ be a semigroup, let $I$ be an ideal of $S,$ and suppose that every right ideal of $I$ is decomposable.  Then $S$ is right noetherian if and only if both $I$ and $S/I$ are right noetherian.
\end{cor}

Recall that a semigroup is strongly right noetherian if its poset of right congruences satisfies the ascending chain condition.  The following result, due to Kozhukhov, describes the non-null principal factors of a strongly right noetherian semigroup.

\begin{prop}\cite[Lemma 1.3]{Kozhukhov:1980}
\label{prop:pf}
Any (0-)simple principal factor of a strongly right noetherian semigroup is completely (0-)simple and has only finitely many $\ar$-classes.
\end{prop}

From Proposition \ref{prop:pf} and Corollary \ref{cor:minrightideals} we immediately deduce:

\begin{cor}
Every non-null principal factor of a strongly right noetherian semigroup is right noetherian.
\end{cor}

\begin{cor}
\label{cor:pf}
Let $S$ be a semigroup with an ideal $I$ such that for every $\jay$-class $J\subseteq I$ the principal factor of $J$ is either simple or 0-simple.  If $S$ is strongly right noetherian, then $I$ is regular and hence right noetherian.
\end{cor}

\begin{proof}
The ideal $I$ is a union of $\jay$-classes.  For every $\jay$-class $J\subseteq I,$ its principal factor is either completely simple or completely 0-simple by Proposition \ref{prop:pf}.  It follows that element of $I$ is regular (in $I$), so $I$ is a regular semigroup.  Hence, by Proposition \ref{prop:reg,rn}, $I$ is right noetherian.
\end{proof}

A semigroup is said to be {\em semisimple} if each of its principal factors is simple or 0-simple.  If a semigroup has a null principal factor, then the non-zero elements of that principal factor are not regular.  Thus regular semigroups are semisimple.  This fact, together with Corollary \ref{cor:pf}, yields:

\begin{cor}
Let $S$ be a strongly right noetherian semigroup.  Then $S$ is semisimple if and only if it is regular, in which case every ideal of $S$ is right noetherian.
\end{cor}

\begin{rem}
Ideals, indeed kernels, of strongly right noetherian (regular) semigroups need not be strongly right noetherian; see \cite[Example 6.5 and Proposition 6.6]{Miller:2020}.
\end{rem}

We end this subsection with some results that will be useful in the next subsection.

\begin{lemma}
\label{lem:rq,wrn}
Let $S$ be a semigroup and let $I$ be an ideal of $S.$  Then $S/I$ is right noetherian if and only if the $S$-act $S_S/I_S$ is noetherian.
\end{lemma}

\begin{proof}
It can easily seen that a subset of $S/I$ is a right ideal of $S/I$ if and only if it is a subact of $S_S/I_S$, and that finite generation is preserved in both directions.
\end{proof}

\begin{cor}
\label{cor:idealext,wrn}
Let $S$ be a semigroup and let $I$ be an ideal of $S.$  Then $S$ is right noetherian if and only if $S/I$ is right noetherian and (the $S$-act) $I_S$ is noetherian.
\end{cor}

\begin{proof}
If $S$ is right noetherian, then so is $S/I$ by Lemma \ref{lem:quotient}.  Since $S_S$ is noetherian, the subact $I_S$ is noetherian by Proposition \ref{prop:acts}.  The converse follows from Lemma \ref{lem:rq,wrn} and Proposition \ref{prop:idealext,wrn}.
\end{proof}

Recalling Construction \ref{con}, an argument similar to the proof of Proposition \ref{prop:con,ACCPR} yields:

\begin{prop}
\label{prop:con,wrn}
Let $S$ be a semigroup, let $A$ be an $S$-act, and let $U=\mathcal{U}(S, A).$  Then $U$ is right noetherian if and only if $S$ is right noetherian and $A$ is noetherian.
\end{prop}

\subsection{Minimal and 0-minimal ideals}

From now on we focus on minimal and 0-minimal ideals.  We begin by exhibiting an example of a right noetherian semigroup with a kernel that is not right noetherian.

\begin{ex}
\label{non-wrn min ideal}
Let $S$ be the semigroup defined by the presentation 
$$\langle a, b\,|\,ab^2=b, aba=a^2b\rangle.$$
Corresponding to the above presentation, we have a rewriting system on $\{a, b\}$ consisting of the rules $ab^2\to b$ and $aba\to a^2b.$
It is straightforward to check that this rewriting is complete (i.e.\ noetherian and confluent) and hence yields the following set of normal forms for $S$:
$$\{a^i, b^ia^j, b^ja^ib : i>0, j\geq 0\};$$
that is, the set of all the words over $\{a, b\}$ that do not contain $ab^2$ or $aba$ as a subword.  For more information about rewriting systems, one may consult \cite{Book:1993} for instance.

Let $A=\langle a\rangle\cong\mathbb{N}.$  We have that
$$a^i(b^ia^j)b^{j+1}=a^ib^{i+1}=b,$$
and hence $a^i(b^ia^jb)b^j=b.$  Thus $S\!\setminus\!A$ is the $\jay$-class of $b.$  Since $S\!\setminus\!A$ is an ideal of $S,$ we conclude that it is the kernel $\K=\K(S).$\par
\vspace{0.5em}
(1) {\em $S$ is right noetherian.}\par
Since $S/\K\cong\mathbb{N}\cup\{0\}$ is right noetherian, by Corollary \ref{cor:idealext,wrn} it suffices to prove that $\K_S$ is noetherian.  So, let $I_S\subseteq\K_S$ be a subact of $S_S.$  We shall prove that $I_S$ is finitely generated.  
Let $i_0\in\mathbb{N}$ be minimal such that $b^{i_0}\in I_S.$  If there exist $j\in\mathbb{N}$ such that $b^{i_0-1}a^j\in I_S,$ let $j_0$ be the minimal such $j$ and set $Y=\{b^{i_0-1}a^{j_0}\}$; otherwise, let $Y=\emptyset.$  If there exist $k\in\mathbb{N}$ such that $b^{i_0-1}a^kb\in I_S,$ let $k_0$ be the minimal such $k$ and set $Z=\{b^{i_0-1}a^{k_0}b\}$; otherwise, let $Z=\emptyset.$
We claim that $I_S$ is generated by $\{b^{i_0}\}\cup Y\cup Z.$  So, let $s\in I_S.$  There are two cases to consider.\par
\textit{Case 1}: $s=b^ia^j$ for some $i>0$ and $j\geq 0.$\par
If $i\geq i_0,$ then $s=b^{i_0}b^{i-i_0}a^j\in b^{i_0}S^1.$
Suppose then that $i<i_0.$  Now $b^{i+1}=sb^{j+1},$ so $i+1\geq i_0$ and hence $i=i_0-1.$  It follows that $j\geq j_0,$ and hence $s=b^{i_0-1}a^{j_0}a^{j-j_0}\in YS^1.$\par
\textit{Case 2}: $s=b^ia^jb$ for some $i\geq 0$ and $j>0.$\par
If $i\geq i_0,$ then $s\in b^{i_0}S,$ so assume that $i<i_0.$  We have that $b^{i+1}=sb^j\in I_S,$ so $i=i_0-1$ and $j\geq k_0.$  Thus $s=b^{i_0-1}a^{k_0}ba^{j-k_0}\in ZS^1.$\par
%\vspace{0.5em}
%(2) {\em $S$ is not right noetherian.}\par
%Let $\sim$ be the right congruence on $S$ generated by the set $X=\{(b^{2i}a, b^{2i-1}) : i\in\mathbb{N}\}.$\par
%We show that $X$ is a minimal generating set for $\sim\!.$  Indeed, let $i\geq 1,$ and consider the application of a pair $(b^{2j}a, b^{2j-1})$ to $b^{2i}a$ where $j\neq i.$
%If $b^{2i}a=b^{2j}au,$ then $i>j$ and $u=b^{2(i-j)+1}a,$ so $$b^{2i}a=b^{2j}au\sim b^{2j-1}u=b^{2i}a.$$
%If $b^{2i}a=b^{2j-1}u,$ then $j<i$ and $u=b^{2(i-j)+1}a,$ so $$b^{2i}a=b^{2j-1}u\sim b^{2j}au=b^{2j}b^{2(i-j)}a=b^{2i}a.$$
%It follows that the pair $(b^{2i}a, b^{2i-1})$ cannot be obtained through applications of pairs from $X\!\setminus\!\{(b^{2i}a, b^{2i-1})\},$ so $X$ is a minimal generating set for $\sim\!.$\par
\vspace{0.5em}
(2) {\em $\K$ is not right noetherian.}\par
We claim that the infinite set $\{(ba^i)\K^1 : i\geq 0\}$ is an antichain of principal right ideals of $\K,$ and hence $\K$ is not right noetherian by Corollary \ref{cor:principal}.  Indeed, consider $ba^iu$ where $u\in\K.$\par
Suppose first that $u=b^ma^n$ for some $m\geq 1$ and $n\geq 0.$  
If $i<m,$ then $ba^iu=b^{m-i+1}a^n$ and $m-i+1\geq 2.$
If $i\geq m,$ then $ba^iu=ba^{i-m+n+1}b.$\par
Now suppose that $u=b^ma^nb$ for some $m\geq 0$ and $n\geq 1.$
If $m=0,$ then $ba^iu=ba^{i+n}b.$  If $1\leq m\leq i,$ then
$$ba^iu=ba^{i-m+1}ba^nb=ba^{i-m+1+n}b^2=ba^{i-m+n}b.$$
Finally, if $m>i$ then $ba^iu=b^{m-i+1}a^nb.$\par
In any case, in view of the normal form for $\K,$ we conclude that $ba^iu\neq ba^j$ for any $j\in\mathbb{N}.$  It follows that $(ba^i)\K^1$ and $(ba^j)\K^1$ are incomparable whenever $i\neq j.$
\end{ex}

The next two results show that in a right noetherian semigroup with minimal one-sided ideals, the kernel is also right noetherian.

\begin{prop}
\label{prop:minrightideal}
Let $S$ be a semigroup with at least one minimal right ideal, and let $\K=\K(S).$  If $S$ is right noetherian, then $\K$ has finitely many $\ar$-classes (of itself), and hence $\K$ is right noetherian.
\end{prop}

\begin{proof}
The kernel $\K$ is the union of all the minimal right ideals of $S.$  By \cite[Theorem 2.4]{Clifford:1948}, each of these minimal right ideals is a minimal right ideal of $\K.$  Moreover, due to their minimality, they form an antichain of principal right ideals of $S.$  Hence, by Corollary \ref{cor:principal}, $\K$ is the union of finitely many minimal right ideals.  Hence, by Corollary \ref{cor:minrightideals}, $\K$ is right noetherian.
\end{proof}

\begin{prop}
\label{prop:minleftideal}
Let $S$ be a semigroup with at least one minimal left ideal, and let $\K=\K(S).$  If $S$ is right noetherian, then $\K$ is completely simple and right noetherian (and hence has finitely many $\ar$-classes).
\end{prop}

\begin{proof}
Since $S$ satisfies ACCPR, the kernel $\K$ is completely simple by Theorem \ref{thm:min_left_ideal}.  Since $\K$ is a regular subsemigroup of $S,$ it is right noetherian by Proposition \ref{prop:reg,rn}, and hence $\K$ has finitely many $\ar$-classes by Corollary \ref{cor:minrightideals}.
\end{proof}

\begin{cor}
Let $S$ be a semigroup with a minimal one-sided ideal, and let $\K=\K(S).$  Then $S$ is right noetherian if and only if both $\K$ and $S/\K$ are right noetherian.
\end{cor}

\begin{proof}
The forward direction follows from Lemma \ref{lem:quotient} and Propositions \ref{prop:minrightideal} and \ref{prop:minleftideal}, and the reverse implication follows from Proposition \ref{prop:idealext,wrn}.
\end{proof}

The remainder of this section concerns semigroups with zero.  The following example demonstrates that a right noetherian semigroup can have a right/left socle that is not right noetherian.  

\begin{ex}
(1) Let $S$ be any right noetherian semigroup, let $A$ be a noetherian $S$-act (such as $S_S$), and let $U=\mathcal{U}(S, A).$  Then $U$ is right noetherian by Proposition \ref{prop:con,wrn}.  For each $a\in A,$ the set $\{x_a, 0\}$ is a null 0-minimal left ideal of $U,$ and $\Sigma^l(U)=\{x_a : a\in A\}\cup\{0\}.$  If $A$ is infinite, then $\Sigma^l(U)$ is not right noetherian; indeed, any infinite null semigroup is not right noetherian by Corollary \ref{cor:minrightideals}.

(2) Let $S$ be the free commutative semigroup on two generators $y$ and $z.$  Let $A=\{a_i : i\in\mathbb{Z}\}$ be the $S$-act with action given by $a_i\cdot y^jz^k=a_{i+j-k}.$  It is easy to see that $A$ has no proper subacts, and hence $A$ is noetherian.  Since $S$ is right noetherian, we have that $U=\mathcal{U}(S, A)$ is right noetherian by Proposition \ref{prop:con,wrn}.  We have that $R=\{x_a : a\in A\}\cup\{0\}$ is a null 0-minimal right ideal of $U,$ and $\Sigma^r(U)=R$ is not right noetherian.
\end{ex}

The following result provides a necessary and sufficient condition for a 0-minimal right ideal to be right noetherian.

\begin{thm}
\label{thm:0-minrightideal}
Let $R$ be a 0-minimal right ideal of a semigroup $S.$  Then $R$ is right noetherian if and only if the set $\{a\in R\!\setminus\!\!\{0\} : aR=0\}$ is finite.
\end{thm}

\begin{proof}
For each $a\in R,$ we have either $aR=0$ or $aR=R.$  Thus, if $a, b\in R\!\setminus\!\!\{0\}$ with $a\subseteq bR,$ then $bR=R.$  It follows that $R$ satisfies ACCPR.  Thus, by Corollary \ref{cor:principal}, $R$ is right noetherian if and only if it has no infinite antichain of principal right ideals.  Now, for any $a, b\in R\!\setminus\!\!\{0\}$ with $a\neq b,$ the principal right ideals $aR^1$ and $bR^1$ are incomparable if and only if $b\notin aR$ and $a\notin bR$ if and only if $aR\neq R$ and $bR\neq R$ if and only if $aR=bR=0.$  The result now follows.
\end{proof} 

Completely 0-simple semigroups have the following well-known representation, due to Rees.  Let $G$ be a group, let $I$ and $J$ be non-empty sets, and let $P=(p_{ji})$ be a $J\times I$ matrix over $G^0$ in which every row and column contains at least one element of $G.$  The {\em Rees matrix semigroup with zero over $G$ with respect to $P$} is the semigroup $\mathcal{M}^0(G; I, J; P)$ with universe $(I\times G\times J)\cup\{0\}$ and multiplication given by
$$(i, g, j)(k, h, l)=
\begin{cases}
(i, gp_{jk}h, l)&\text{ if }p_{jk}\in G\\
0&\text{ otherwise,}
\end{cases}
\hspace{1em}0(i, g, j)=(i, g, j)0=0^2=0.$$
The 0-minimal right ideals of $\mathcal{M}^0(G; I, J; P)$ are the sets $R_i=(\{i\}\times G\times J)\cup\{0\}$ ($i\in I$).  From Theorem \ref{thm:0-minrightideal} we deduce:

\begin{cor}
Let $S=\mathcal{M}^0(G; I, J; P)$ be a completely 0-simple semigroup.  Then a 0-minimal right ideal $R_i=(\{i\}\times G\times J)\cup\{0\}$ of $S$ is right noetherian if and only if the set $\{(g, j)\in G\times J : p_{ji}=0\}$ is finite.
\end{cor}

\begin{cor}
\label{cor:comp_0-simple}
Let $S=\mathcal{M}^0(G; I, J; P)$ be a completely 0-simple semigroup where $G$ is infinite.  Then a 0-minimal right ideal $R_i=(\{i\}\times G\times J)\cup\{0\}$ of $S$ is right noetherian if and only if $p_{ji}\in G$ for all $j\in J.$
\end{cor}

\begin{rem}
Let $S=\mathcal{M}^0(\mathbb{Z}; I, I; P)$ where $|I|=2$ and
$P=\begin{pmatrix}
1 & 0\\
0 & 1 
\end{pmatrix}$.
Then $S$ is strongly right noetherian by \cite[Corollary 2.2]{Kozhukhov:2003}, but neither of its two 0-minimal right ideals are right noetherian by Corollary \ref{cor:comp_0-simple}.
\end{rem}

Although the right socle of a right noetherian semigroup need not be right noetherian itself, it is necessary that the globally idempotent part of the right socle be right noetherian.

\begin{prop}
\label{prop:0-min_right_ideal,wrn}
Let $S=S^0$ be a right noetherian semigroup.  Then $S$ has finitely many 0-minimal right ideals.  Moreover, if $S$ has a globally idempotent 0-minimal right ideal, then the globally idempotent part $B^r$ of $\Sigma^r=\Sigma^r(S)$ is a union of finitely many 0-minimal right ideals of itself, and hence $B^r$ is right noetherian. 
\end{prop}

\begin{proof}
The set of 0-minimal right ideals of $S,$ if non-empty, is an antichain of principal right ideals of $S.$  Therefore, since $S$ is right noetherian, it has finitely many 0-minimal right ideals by Corollary \ref{cor:principal}.  

Now suppose that $S$ has a globally idempotent 0-minimal right ideal.  By Theorem \ref{thm:rightsocle}, there exists a set $\{R_i : i\in I\}$ of globally idempotent 0-minimal right ideals of $S$ such that $B^r$ is the 0-direct union of the $B^{R_i}$ $(i\in I).$  Then $I$ is finite, and it follows from Theorem \ref{thm:SR}(4) that each $B^{R_i}$ is a union of 0-minimal right ideals of itself.  It then clearly follows that $B^r$ is a union of finitely many 0-minimal right ideals of itself.  Hence, by Corollary \ref{cor:minrightideals}, $B^r$ is right noetherian.
\end{proof}

\begin{cor}
Let $S=S^0$ be a semigroup, and let $\Sigma^r=\Sigma^r(S).$  Then the following are equivalent:
\begin{enumerate}
\item $S$ is right noetherian;
\item $S$ has finitely many 0-minimal right ideals and $S/\Sigma^r$ is right noetherian.
\end{enumerate}
\end{cor}

\begin{proof}
(1)$\Rightarrow$(2) follows immediately from Proposition \ref{prop:0-min_right_ideal,wrn} and Lemma \ref{lem:quotient}.

(2)$\Rightarrow$(1).  The right socle $\Sigma^r$ contains only finitely many right ideals of $S$; equivalently, $\Sigma^r_S$ contains only finitely many subacts of $S_S.$  Thus $\Sigma^r_S$ is noetherian.  Since $S/\Sigma^r$ is right noetherian, we have that $S$ is right noetherian by Corollary \ref{cor:idealext,wrn}.
\end{proof}

\begin{cor}
Let $S=S^0$ be a semigroup without null 0-minimal ideals, and let $\Sigma^r=\Sigma^r(S).$  Then the following are equivalent:
\begin{enumerate}
\item $S$ is right noetherian;
\item $\Sigma^r$ is a union of finitely many 0-minimal right ideals of itself, and $S/\Sigma^r$ is right noetherian;
\item both $\Sigma^r$ and $S/\Sigma^r$ are right noetherian.
\end{enumerate}
\end{cor}

\begin{proof}
(1)$\Rightarrow$(2).  We have $\Sigma^r=B^r$, so $\Sigma^r$ is a union of finitely many 0-minimal right ideals of itself by Proposition \ref{prop:0-min_right_ideal,wrn}.   By Lemma \ref{lem:quotient}, $S/\Sigma^r$ is right noetherian. 

(2)$\Rightarrow$(3) follows from Corollary \ref{cor:minrightideals}, and (3)$\Rightarrow$(1) follows from Proposition \ref{prop:idealext,wrn}.
\end{proof}

The following result is an analogue of Proposition \ref{prop:minleftideal} for 0-minimal left ideals.

\begin{prop}
\label{prop:0-min_left_ideal,wrn}
Let $S=S^0$ be a semigroup with a globally idempotent 0-minimal left ideal $L.$  If $S$ is right noetherian, then the globally idempotent part $B^L$ of $LS$ is completely 0-simple and right noetherian (and hence has finitely many $\ar$-classes).  Moreover, $L$ is right noetherian.
\end{prop}

\begin{proof}
Since $S$ satisfies ACCPR, $B^L$ is completely 0-simple by Corollary \ref{cor:0-min_left_ideal}.  Therefore, $B^L$ is right noetherian by Proposition \ref{prop:reg,rn}, and hence $B^L$ has finitely many $\ar$-classes by Corollary \ref{cor:minrightideals}.  Since $L$ is contained in $B_L$, which is regular, $L$ is right noetherian by Corollary \ref{cor:leftideal,wrn}.
\end{proof}

We now characterise the property of being right noetherian in terms of the left socle.

\begin{thm}
\label{thm:leftsocle,wrn}
Let $S=S^0$ be a semigroup, and let $\Sigma^l=\Sigma^l(S).$  Then the following are equivalent:
\begin{enumerate}
\item $S$ is right noetherian;
\item $B^l$ is either 0 or the 0-direct union of finitely many completely 0-simple semigroups that each have finitely many $\ar$-classes, the $S$-act $A^l_S$ is noetherian, and $S/\Sigma^l$ is right noetherian.
\item both $\Sigma^l$ and $S/\Sigma^l$ are right noetherian.
\end{enumerate}
\end{thm}

\begin{proof}
(1)$\Rightarrow$(2).  By Corollary \ref{cor:idealext,wrn}, $A^l_S$ is noetherian and $S/\Sigma^l$ is right noetherian.  Suppose that $B^l\neq 0.$  Then, by the left-right dual of Theorem \ref{thm:rightsocle}, there exists a set $\{L_i : i\in I\}$ of globally idempotent 0-minimal left ideals such that $B^l$ is a 0-direct union of $B_i$ $(i\in I),$ where $B_i=B^{L_i}.$  By Proposition \ref{prop:0-min_left_ideal,wrn}, each $B_i$ is completely 0-simple and has finitely many $\ar$-classes.  For each $i\in I,$ let $e_i$ be a non-zero idempotent in $B_i.$  We cannot have $e_j\in e_iS$ for any $i\neq j,$ for that would imply that $e_ie_j=e_j,$ contradicting the fact that $B_iB_j=0.$  Thus $\{e_iS^1 : i\in I\}$ is an antichain of principal right ideals of $S,$ and hence $I$ is finite by Corollary \ref{cor:principal}.

(2)$\Rightarrow$(3).  We have that $\Sigma^l/A_l\cong B^l$ is right noetherian by Corollary \ref{cor:minrightideals}.  Therefore, since $A^l_S$ is noetherian, $\Sigma^l$ is right noetherian by Corollary \ref{cor:idealext,wrn}. 

(3)$\Rightarrow$(1) follows from Proposition \ref{prop:idealext,wrn}.
\end{proof}

\begin{cor}
Let $S=S^0$ be a semigroup without null 0-minimal ideals, and let $\Sigma^l=\Sigma^l(S).$  Then the following are equivalent:
\begin{enumerate}
\item $S$ is right noetherian;
\item $\Sigma^l$ is either 0 or the 0-direct union of finitely many completely 0-simple semigroups that each have finitely many $\ar$-classes, and $S/\Sigma^l$ satisfies ACCPR.
\end{enumerate}
\end{cor}

We now find several equivalent characterisations for a semigroup $S=\Sigma^l(S)$ to be right noetherian.

\begin{thm}
\label{thm:S=leftsocle}
Let $S=S^0$ be a semigroup such that $S=\Sigma^l=\Sigma^l(S),$ and let $\Sigma^r=\Sigma^r(S).$  Then the following are equivalent: 
\begin{enumerate}
\item $S$ is right noetherian;
\item $A^l$ is finite, and $B^l$ is either 0 or the 0-direct union of finitely many completely 0-simple semigroups that each have finitely many $\ar$-classes;
\item $\Sigma^r$ is either a finite null semigroup or the 0-direct union of a finite null semigroup and finitely many completely 0-simple semigroups that each have finitely many $\ar$-classes, and either $\Sigma^r=S$ or $S/\Sigma^r$ is the 0-direct union of finitely many completely 0-simple semigroups that each have finitely many $\ar$-classes;
\item $S$ has finitely many $\ar$-classes.
\end{enumerate}
\end{thm}

\begin{proof}
(1)$\Rightarrow$(2).  Given Theorem \ref{thm:leftsocle,wrn}, we only need to prove that $A^l$ is finite.  By Lemma \ref{lem:S=left_socle}, either $A^l\neq 0$ or $\{a, 0\}$ is a 0-minimal right ideal of $\Sigma^l$ for each $a\in A^l\!\setminus\!\{0\}.$  Since $\Sigma^l$ is right noetherian, it has only finitely many 0-minimal right ideals by Proposition \ref{prop:0-min_right_ideal,wrn}, so $A^l$ is finite.

The proof of (2)$\Rightarrow$(3) is essentially the same as that of (2)$\Rightarrow$(3) of Theorem \ref{thm:S=left_socle}.  (3)$\Rightarrow$(4) is obvious, and (4)$\Rightarrow$(1) follows from Corollary \ref{cor:R-classes}.
\end{proof}

\begin{cor}
Let $S=S^0$ be a semigroup such that $S=\Sigma^l(S)=\Sigma^r(S).$  Then $S$ is right noetherian if and only if it is either a finite null semigroup or the 0-direct union of a finite null semigroup and finitely many completely 0-simple semigroups that each have finitely many $\ar$-classes.
\end{cor}

We now present an example to illustrate Theorem \ref{thm:S=leftsocle}, and to demonstrate that a right noetherian semigroup can be the union, but not 0-direct union, of its 0-minimal left ideals. 

\begin{ex}
Let $V$ be the 0-disjoint union of two completely 0-simple semigroups $S$ and $T,$ each with finitely many $\ar$-classes, and let $x$ be an element disjoint from $V.$  Let $U=V\cup\{x\},$ and define a multiplication on $U,$ extending that on $V,$ as follows:
$$sx=x\,\text{ and }\,xv=x^2=tx=0x=0$$
for all $s\in S,$ $t\in T$ and $v\in V.$  It is straightforward to show that $U$ is a semigroup under this multiplication.  It is easy to see that the 0-minimal left ideals of $U$ are $\{x, 0\}$ and the 0-minimal left ideals of $S$ and $T.$  Thus $U=\Sigma^l(U),$ where $A^l=\{x, 0\}$ and $B^l=V,$ and $U$ is right noetherian by Theorem \ref{thm:S=leftsocle}.  The 0-minimal right ideals of $U$ are $\{x, 0\}$ and the 0-minimal right ideals of $T,$ and $\Sigma^r$ is the 0-direct union of $\{x, 0\}$ and $T.$  On the other hand, $U$ is not the 0-direct union of its 0-minimal left ideals (since $sx=x$ for all $s\in S$).
\end{ex}

\begin{rem}
Let $S$ be any right simple semigroup, and let $U=\mathcal{U}(S, S_S).$  Then $S$ is right noetherian by Proposition \ref{prop:con,wrn}.  It is easy to see that $U$ is the union of its two 0-minimal right ideals, $I=\{x_s : s\in S\}\cup\{0\}$ and $S^0.$  Thus $U=\Sigma^r(U),$ where $I=A^r$ and $S^0=B^r.$  Since $x_st=x_{st}$ for all $s, t\in S,$ the semigroup $U$ is not the 0-direct union of $I$ and $S^0.$
\end{rem}

\section*{Acknowledgments}
This work was supported by the Engineering and Physical Sciences Research Council [EP/V002953/1].  The author thanks the referee for a number of helpful suggestions.

\end{document}